\documentclass[12pt]{amsart}
\pdfoutput=1
\usepackage[margin=1.25in]{geometry}
\usepackage{amsmath}
\usepackage{graphicx}
\usepackage{latexsym}
\usepackage{color}
\usepackage{amscd}
\usepackage[all]{xy}
\usepackage{enumerate}
\usepackage{soul}
\usepackage{comment}
\usepackage{epsfig}
\usepackage{epstopdf}
\usepackage{tikz}
\usepackage{tikz-cd}
\usepackage{marginnote}
\usepackage{lipsum}
\usepackage{mathtools}
\usepackage{amssymb}
\usepackage{caption} 
\usepackage{subcaption}

\usepackage{multirow}

\newtheorem{thm}{Theorem}

\newtheorem{prop}[thm]{Proposition}

\newtheorem{theorem}[thm]{Theorem}
\newtheorem{lemma}[thm]{Lemma}
\newtheorem{corollary}[thm]{Corollary}
\newtheorem{proposition}[thm]{Proposition}

\newtheorem*{mthma}{Theorem~A}
\newtheorem*{mthmb}{Theorem~B}
\newtheorem*{mthmc}{Theorem~C}

\theoremstyle{definition}
\newtheorem{defn}[thm]{Definition}
\newtheorem*{definition*}{Definition}

\newtheorem{rmk}[thm]{Remark}

\newtheorem{remark}[thm]{Remark}

\newcommand{\CPb}{\overline{\mathbb{CP}}{}^{2}}
\newcommand{\CP}{{\mathbb{CP}}{}^{2}}
\newcommand{\CPo}{{\mathbb{CP}}{}^{1}}

\newcommand{\RP}{{\mathbb{RP}}{}^{2}}

\newcommand{\R}{\mathbb{R}}

\newcommand{\N}{\mathbb{N}}
\newcommand{\Z}{\mathbb{Z}}

\newcommand{\M}{\operatorname{Mod}}
\newcommand{\D}{\operatorname{Diff{^+}}}

\newcommand{\K}{{\rm K3}}

\def \x {\times}

\renewcommand{\phi}{\varphi}
\newcommand{\pr}{\text{pr}}

\newcommand{\coker}[1]{\text{coker}(#1)}

 \def\R{{\mathbb{R}}}
 \def\Z{{\mathbb{Z}}}
 
 \def\N{{\mathbb{N}}}

\begin{document}

\title[Nielsen realization and projective twists]{Nielsen realization in dimension four \\ and projective twists}

\author[M. Arabadji]{Mihail Arabadji}
\address{Department of Mathematics and Statistics, University of Massachusetts, Amherst, MA 01003, USA}
\email{marabadji@umass.edu}

\author[R. \.{I}. Baykur]{R. \.{I}nan\c{c} Baykur}
\address{Department of Mathematics and Statistics, University of Massachusetts, Amherst, MA 01003, USA}
\email{inanc.baykur@umass.edu}

\maketitle 

\begin{abstract}
We demonstrate the existence of numerous \emph{non-spin} $4$--manifolds for which the smooth \emph{Nielsen realization problem} fails; namely,  there exist finite subgroups of their mapping class groups that cannot be realized by any group of diffeomorphisms. This  extends and complements recent results for spin $4$--manifolds.  Our examples span virtually all possible intersection forms,  both even and odd, indefinite and definite,  and include many irreducible $4$--manifolds. To derive these examples, we study  multi-twists, projective twists, and multi-reflections, which are all mapping classes supported around collections of embedded spheres and projective planes. Our obstructions to Nielsen realization are based on the work of Konno.  We investigate projective twists in further detail, and notably,  employ them to show that, for many closed symplectic $4$--manifolds, the symplectic Torelli group is \emph{not} generated by squared Dehn twists.
\end{abstract}

\section{Introduction}

Let $X$ be a closed oriented smooth $4$--manifold and let $\M(X)$  be its smooth mapping class group,  which consists of smooth isotopy classes of orientation-preserving diffeomorphisms.  The \textit{Nielsen realization problem} \cite{Nielsen},  which famously has an affirmative answer for surfaces \cite{Kerchkoff},  asks if the natural surjection $\textrm{Diff}^+(X) \to \M(X)$ has a section over any finite subgroup of $\M(X)$.

Recent results by Kronheimer--Mrowka \cite{KM},  Baraglia--Konno  \cite{BaragliaKonno},  Lin \cite{Lin},   Farb-Looijenga \cite{FL} and Konno \cite{Konno} demonstrated that the Nielsen realization fails for many spin $4$--manifolds.  Along with the classical examples of Raymond and Scott \cite{RaymondScott}, which are nilmanifolds constructed as the mapping tori of the $3$--torus,  all the examples of $4$--manifolds  for which the Nielsen realization fails admit spin structures, and thus have even intersection forms.  On the other hand,  Lee recently showed that many finite subgroups of the topological mapping class group of small  rational surfaces \textit{are}  realized by finite groups of diffeomorphisms \cite{Lee}.  These  cumulatively yield the question of to what extend the Nielsen realization is governed by the intersection form; e.g. ~\cite[Problem~5.2]{KonnoProblems}.  

The main goal of this note is to show that the failure of the Nielsen realization problem is virtually independent of the intersection form.

\begin{mthma}
Given any non-zero integer $s$,  for all but finitely many indefinite unimodular integral quadratic form $Q$ with signature $\sigma(Q)=s$,  there is a closed oriented smooth non-spin $4$--manifold $X$ with finite fundamental group,  intersection form $Q_X \cong Q$,  and an order two  subgroup of $\M(X)$ that does not lift to $\textrm{Diff}^+(X)$.  The examples include irreducible $4$--manifolds realizing infinitely many intersection forms.
\end{mthma}
 
Our results are of  slightly different flavor in the definite case:

\begin{mthmb}
For any standard definite unimodular integral quadratic form $Q$ with rank at least five,  there is a closed oriented smooth $4$--manifold $X$ with finite fundamental group,  intersection form $Q_X \cong Q$,   and an order two subgroup of $\M(X)$ which does not lift to $\textrm{Diff}^+(X)$.
\end{mthmb}
  
In the proofs of these theorems,  we employ a small variety of mapping classes which are compactly supported around certain embedded spheres and projective planes in an oriented $4$--manifold.  In Theorem~A, we use \emph{multi-twists},  simultaneous Dehn twists along a collection of disjoint spheres with normal Euler numbers $\pm 2$, whereas in Theorem~B, we use \emph{multi-reflections} along a collection of disjoint spheres with normal Euler numbers $\pm 1$.  The obstructions to the realizability of these mapping classes we consider are derived from the inspiring work of Konno in \cite{Konno},  by carefully cooking up our  examples so that their universal covers are spin $4$--manifolds subject to the constraints  from involutive Seiberg-Witten theory.   

Different examples of non-spin $4$--manifolds for which the Nielsen realization fails were given very recently by Konno, Miyazawa and Taniguchi \cite{KonnoEtal}. The authors showed that for certain non-minimal $4$--manifolds,  a  \emph{reflection} along a $2$--sphere with normal Euler number $-1$, which is an involution locally modeled as complex conjugation around the exceptional sphere,  constitutes an order two mapping class that does not lift.  Akin to Konno's earlier work in  \cite{Konno}, the constraints given in \cite{KonnoEtal} allow us to build further non-liftable examples by studying \emph{multi-reflections}, which are simultaneous reflections along a  collection of disjoint $2$--spheres with normal Euler numbers $\pm 1$.  We have thus opted to revise our earlier results in the definite case to  include these examples,  which both simplified our proofs and improved the overall statement of Theorem~B. 

Moreover, we delve into a study \emph{projective twists}, which are analogs of Dehn twists performed along projective planes with normal Euler number $\pm 1$ and are smooth mapping classes that may be of independent interest.  We observe that, unlike the Dehn twist along a sphere \cite{Seidel}, locally the projective twist has infinite order; see Proposition~\ref{prop:PT}.  We also show that a projective twist along $R$ in $X$ may or may not be trivial in $\M(X)$,  depending on the topology of the pair $(X, R)$; see Proposition~\ref{prop:essential} and Remark~\ref{rk:trivialPT}.  Although we cannot determine the order of projective twists in general, we propose them as another potential source of non-liftable subgroups of the smooth mapping class groups of $4$--manifolds; in particular, we present a curious pair $(X,R)$ where $X$ is a definite irreducible $4$--manifold; see Remark~\ref{rk:rank4}. 

A noteworthy outcome of our investigation of projective twists is a negative answer to the following question attributed to Donaldson \cite{SheridanSmith, LiEtal}: \textit{Is the symplectic Torelli mapping class group of a closed symplectic $4$--manifold generated by squared  Dehn twists along Lagrangian spheres?} Notably, the answer is affirmative for positive rational symplectic surfaces \cite{LiEtal}; however,  we show that this is not the case in general:

\begin{mthmc}
There is an infinite family of closed symplectic $4$--manifolds, where each symplectic $4$--manifold $(X, \omega)$ contains a Lagrangian projective plane $R$,  such that any odd power of $T_R \in \M(X,\omega)$ is non-trivial  even in  $\M(X)$.  It follows that the symplectic Torelli group $\mathcal{I}(X, \omega)$ is not generated by squared Dehn twists,  even in $\M(X)$.  In these examples,  we can take $X$ with finite $\pi_1(X)$ and of any  non-negative symplectic Kodaira dimension.
\end{mthmc}   

As it is the case in all the examples we produce in this paper,   symplectic $4$--manifolds in Theorem~C are not simply connected,  but can be taken with $\pi_1=\Z_2$.

\bigskip
\noindent \textit{Conventions and outline of the article.} All the manifolds and maps we consider in this article are smooth unless explicitly stated otherwise.  For a compact connected oriented manifold $X$, we denote by $\textrm{Diff}^+(X)$ the group of orientation-preserving diffeomorphisms of $X$,  and when the boundary $\partial X \neq \emptyset$,  we assume they  restrict to identity in a neighborhood of $\partial X$. We denote by $\M(X):=\pi_0(\textrm{Diff}^+(X))$ the (smooth) mapping class group of $X$.  
Given a symplectic form on $\omega$ on $X$,  let $\textrm{Symp}\,(X, \omega)$ be the  subgroup of $\D(X)$ stabilizing the form $\omega$.  We call the subgroup of the \emph{symplectic mapping class group} $\M(X, \omega):= \pi_0(\textrm{Symp}\,(X, \omega))$  that consists of elements that act trivially on $H_*(X)$  the  \emph{symplectic Torelli group},  and  denote it  by $\mathcal{I}(X, \omega)$.  

 We provide all the formal definitions and discuss the  properties of multi-twists,  projective twists and multi-reflections in Section~2.  Obstructions coming from involutive Seiberg-Witten theory are gathered in Section~3.  In Sections 4,  5 and 6, we prove Theorems~A,  B and~C,  respectively.  Their proofs are going to involve a potpourri of smooth,  symplectic and complex  constructions in real dimension four (albeit with mainly smooth outcomes) to produce the desired examples of $4$--manifolds with particular collections of spheres and projective planes in them and particular universal finite covers.

\medskip
\section{Multi-twists,  projective twists and multi-reflections}

Here we are going to define the three types of mapping classes that are of interest to us in this article and review their basic properties.

\subsection{Multi-twists} 

Let $\rho\colon S^1\times (TS^2 \setminus S^2) \to TS^2 \setminus S^2$ be a smooth action where $\rho(e^{it},-)$ is the parallel transport of $(u,v)$ by $t$ radians along the geodesic of $v$ on $S^2$.
Let $\psi\colon \R\to \R$ be a smooth function with value $\pi$ around the origin and zero outside of $(-1,1)$. Then
\begin{align*}
\tau_0:TS^2&\to TS^2 \\
(u,v)&\mapsto \rho(e^{i\psi(|v|)},(u,v)) \text{ for $v\neq 0$}\\
(u,0)&\mapsto (-u,0)
\end{align*}
is an orientation-preserving diffeomorphism compactly supported in the unit disk tangent bundle $D\,TS^2$. 

Let $S$ be an embedded sphere in an oriented $4$--manifold $X$ with normal Euler number $-2$, so its normal bundle $\nu S$ is isomorphic to $TS^2$. Thus, there is a compact tubular neighborhood $N$ of $S$ and an orientation-preserving diffeomorpism \linebreak $\phi \colon N  \to D\,TS^2$ such that $\phi(S)$ is the zero-section of $TS^2$.  Extending $\phi^{-1}\circ \tau_0\circ\phi$ by identity on $X \setminus N$, we get a mapping class $T_{S} \in \M(X)$, called the \textit{Dehn twist} along $S$. Since tubular neighborhoods are unique up to isotopy, the mapping class $T_S$ only depends on isotopy class of $S$ in $X$. If $S$ is an embedded sphere in $X$ with normal Euler number $+2$, we can similary define the Dehn twist along $S$ by taking an orientation-reversing \mbox{$\phi \colon N  \to D\,TS^2$}. With this in mind, we can now define a generalization of this mapping class:

\begin{defn} \label{def:MT}
Let $S=\sqcup_{i=1}^k S_i$ be a disjoint union of embedded spheres with normal Euler numbers $\pm 2$ in an oriented $4$--manifold $X$. The \textit{multi-twist} along $S$ is the  mapping class $T_{S}:= T_{S_1}\cdot \ldots \cdot T_{S_k}$ in $\M(X)$.
\end{defn}

Dehn twists along any two components of $S$ can be represented by diffeomorphisms that are supported away from each other,  so they commute.  Thus, the definition  of $T_S$ we have given does not depend on the order of the product $T_{S_1}\cdot \ldots \cdot T_{S_k}$, and in turn, on the labeling of the components of $S$. Note that for any tubular neighborhood $N$ of $S$, the multitwist $T_S$ can be viewed as an element of $\M(N)$.

\begin{remark}
We can also view the Dehn twist $T_S$ along an embedded sphere $S$ of normal Euler number $-2$  as an element of the mapping class group of the disk cotangent bundle $D\, T^*S$,  by identifying the latter with the normal bundle of $S$, which is the preferred definition in  symplectic topology  in \cite{Arnold, Seidel} as $T_S$ along a Lagrangian sphere becomes an element of the symplectic mapping class group. 
\end{remark}

In the next proposition, we list a few  properties of $T_S$, all of which easily extend from those of a Dehn twist along a single sphere:

\begin{proposition} \label{prop:MT}
Let $T_S \in \M(X)$ be the multi-twist along $S=\sqcup_{i=1}^k S_i$ and $N$ a compact tubular neighborhood of $S$.  Set $c_i:=[S_i] \in H_2(X)$ and $\epsilon_i:= -\rm{sign}\,(c_i^2)$. 
\begin{enumerate}[\rm{(}a\rm{)}]
\item $ (T_S)_*(a)=a+ \sum_{i=1}^k \epsilon_i \, (a \cdot c_i) \,c_i $ for any $a \in H_2(X)$.
\item $T_S$ has order $2$ in $\M(N)$,  and therefore also in $\M(X)$. 
\item $T_S$ preserves any spin structure (if it exists) on $X$.
\end{enumerate}
\end{proposition}

\begin{proof}
(a) The action of $T_S$ on $H_2(X)$ is derived inductively from the Picard-Lefschetz formula for the action of a Dehn twist along a single $S_i$, noting that $c_i \cdot c_j = 0$ for each $i \neq j$. Recall that for a Dehn twist along $S_i$ with normal Euler number $-2$,  we have $ (T_{S_i})_*(a)=a+  (a \cdot c_i)c_i $.  When the normal Euler number of $S_i$ is $+2$ instead,  through the same calculation in $X$ with reversed orientation,  we have $ (T_{S_i})_*(a)=a- (a \cdot c_i)c_i $.  Therefore,  $ (T_{S_i})_*(a)=a+ \epsilon_i \, (a \cdot c_i)c_i $,  which is the case for $k=1$.  

As for the inductive step:
\begin{align*}
	 (T_S)_*(a)&=(T_{S_k})_*(a+ \sum_{i=1}^{k-1} \epsilon_i \,  (a \cdot c_i) \,c_i) \\
	 &=(a+ \sum_{i=1}^{k-1} \epsilon_i \, (a \cdot c_i) \,c_i) +\epsilon_k ( (a+\sum_{i=1}^{k-1} \epsilon_i \, (a \cdot c_i) \,c_i) \cdot c_k)c_k\\
	 &=(a+ \sum_{i=1}^{k-1} \epsilon_i \, (a \cdot c_i) \,c_i) + \epsilon_k \, \left(a\cdot c_k\right)c_k\\
	 &=a+ \sum_{i=1}^k \epsilon_i \, (a \cdot c_i) \,c_i.
\end{align*}
 
\noindent (b) The multi-twist $T_S$ is not trivial. It acts non-trivially on $c_1$ by $(T_S)_*(c_1)= - c_1$, where $c_1$ is certainly not $2$--torsion, as $c_1 \cdot c_1 = \pm 2$. On the other hand, it was observed by Kronheimer (and by Seidel in \cite{Seidel} with an explicit isotopy) that the square of the local model $\tau_0 \in \textrm{Diff}^+(D\,TS^2)$ is isotopic to identity rel boundary.  Hence, any Dehn twist has order $2$. By the commutativity of the factors of $T_{S}$, we have $T_S^2 = T_{S_1}^2 \cdot \ldots \cdot T_{S_k}^2=1$. 

\noindent (c) Because a  Dehn twist preserves every  spin structure \cite{Konno},  so does any product of Dehn twists,  and  $T_S$ in particular. 
\end{proof}

\smallskip
\subsection{Projective twists}

Observe that the local model for the Dehn twist respects the antipodal involution $\iota(u,v) = (-u,-v)$ on $T S^2$,  since for any $t$, we have
$$
\rho(e^{i\psi(t)},\iota(u,v)) = \rho(e^{i\psi(t)},\rho(e^{i\pi},(u,v))) \\
= \rho(e^{i\pi },\rho(e^{i\psi(t)},(u,v))) = \iota( \rho(e^{i\psi(t)},(u,v))),
$$
so $\tau_0$ and $\iota$ commute in $\textrm{Diff}^+(D\, TS^2)$. 
Identifying $T\,\RP$ as $TS^2 \, / (u,v)\sim \iota(u,v)$,  we get a map $\tau'_0 \in \textrm{Diff}^+(D\,  T \RP)$ covered by $\tau_0$.

Let $R$ be an embedded real projective place in an oriented $4$--manifold $X$ with normal Euler number $-1$. (This should be a non-orientable bundle of course,  since $w_1(\nu R)= w_1(TX|_R)-w_1(TR)=-w_1(TR)\neq 0$.) So the normal bundle $\nu R$ is isomorphic to $T  \RP$, and there is a compact tubular neighborhood $N'$ of $R$ and an orientation-preserving diffeomorphism $\phi'\colon N' \to D \, T\RP$ such that $\phi'(R)$ is the zero-section of $T \RP$.  	Extending $\phi'^{-1}\circ \tau'_0 \circ\phi'$ by identity on $X \setminus N'$, we get a mapping class $T_{R} \in \M(X)$. Once again, it is easy to see that this definition only depends on the isotopy class of $R$ in $X$.  Moreover, if $R$ were to have normal Euler number $+1$ instead, by taking the orientation-reversing $\phi'$,  we can similarly define $T_R$ in $\M(X)$.

\begin{defn}
Let $R$ be an embedded real projective plane with normal Euler number $\pm 1$ in an oriented $4$--manifold $X$. The \emph{projective twist} along $R$ is the mapping class $T_R$ in $\M(X)$ defined above. 
\end{defn}

\begin{remark} 
Since the projective twist is  induced by an equivariant Dehn twist, we chose to use the same notation for both types of twists.  Generally, for $R=\sqcup_{i=1}^k R_i$ a disjoint union of embedded real projective planes with normal Euler number $\pm 1$, we can define a \textit{projective multi-twist} as $T_R:= T_{R_1} \cdot \ldots \cdot T_{R_k}$ in $\M(X)$. Further, we can define $T_S \in \M(X)$ for any embedded surface $F$ in $X$ which is a disjoint union of  spheres with normal Euler numbers $\pm 2$ and projective planes with normal Euler numbers $\pm 1$, where one simply infers which type of twist to apply along a component based on its (ambient) topology ---whether it is a sphere or projective plane and (the sign of) its normal Euler number. 
\end{remark}

\begin{remark}
Akin to the case of a Dehn twist,  one can also view $T_R$ in the mapping class group of the unit disk cotangent bundle $D \, T^* \RP$ and in fact get an element of the symplectic mapping class group when $R$ is Lagrangian; see \mbox{e.g. \cite{Evans}.}  
\end{remark}

We have the following basic properties for a projective twist.  

\begin{proposition} \label{prop:PT}
	Let $T_R \in \M(X)$ be a projective twist along $R$ and $N$ be a compact tubular neighborhood of $R$. 
	\begin{enumerate}[\rm{(}a\rm{)}]
		\item $T_R$ acts trivially on $H_2(X)$.
		\item $T_R$  has infinite order in $\M(N)$.
		\item $T_R$ preserves any spin structure (if it exists) on $X$.
	\end{enumerate}
\end{proposition}

\begin{proof}
(a) Let $a\in H_2(X)$ be represented by a surface $\Sigma \subset X$ intersecting $R$ transversally.  Also identify a tubular neighborhood $N$ of $R$ with its unit disk normal bundle.  We can assume that $\Sigma \cap N$ is a disjoint union of disk fibers and since $T_R$ is supported in $N$, the image of $\Sigma$ will be determined once we understand the image of a disk fiber.  Consider the commutative diagram

$$\begin{tikzcd}
	H_2(DTS^2,\partial F) \arrow{r}{(\tau_0)_*}  \arrow{d}[left]{q_*} & H_2(DTS^2,\partial F) \arrow{d}[left]{q_*}\\
	H_2(DT R,\partial F')  \arrow{r}[below]{(\tau_0')_*} & H_2(DT R,\partial F') . 
\end{tikzcd}
$$
where $q$ is the double covering map, and $\partial F$ and $\partial F'$ are the boundaries of fixed disk fibers $F\subset DTS^2$ and $F'=q(F)\subset DT R$, respectively (so $F$ is \textit{one of} the disk fibers in $q^{-1}(F')$).  
Following the commuting diagram above,  we get
$$(\tau_0')_*([F'])=({\tau_0'})_* (q_*([F]))=q_* ((\tau_0)_*([F]))=q_*([F] \pm [S^2])=[F'] \pm 2[R]=[F'].$$
Therefore, $T_R(a)=a$.

\noindent (b) We work in $D\, TS^2$ and $D\, T\RP$. For notational simplicity, let us denote the $0$--sections $S^2$ and $\RP$ by $S$ and $R$, these bundles by $DTS$ and $DTR$, 
and the diffeomorphisms in the local models representing these twists by $\tau_S$ and $\tau_R$, respectively. Recall that we have a double cover $q \colon DTS\to DTR$ induced by the involution $\iota(x,y)=(-x,-y)$ on $DTS$ so that $\tau_R \circ q = q \circ \tau_S$. 

First, we are going to show that $\tau_R^2 \nsim id$ in $\D(DTR)$. We will then generalize the argument to any non-zero power of $\tau_R$.

To arrive at a contradiction, let's assume that there is a smooth isotopy $h_t$ from $\tau_R^2$ to $id_{DTR}$. The isotopy $h_t$ lifts uniquely to an $\iota$--equivariant isotopy $\widetilde{h}_t$ in $\D(DTS)$ starting at $\widetilde{h}_0 := \tau_S^2$. Then the deck transformation $\widetilde{h}_1={id}_{DTS}$ and not $\iota$, because $\widetilde{h}_0$ acts trivially on $H_2(DTS)$ whereas $\iota$ maps the non-trivial class $[S]$ to $ -[S]$.

Fix an equator $S^1\subset S$ and pick $\xi_0 \in DTS$ tangent to $S^1$ at some point $p_0\in S^1$ with $|\xi_0|=1$.  Let  $$T:=\{\xi \in DTS^1\subset DTS\; \vert \; \xi \text{ is a parallel transport of } \xi_0 \text{ along } S^1 \}$$ 
and for a fixed projection $\pr:DTS \to S$,  let
$$D_\xi:=\{\pr( \widetilde{h}_t(s\xi)) \; | \; t\in [0,1] \text{ and } s\in [-1,1] \} .$$
so we have a map
\begin{align*}
\varphi\colon T &\to H_2(S, S^1) \\
\xi &\mapsto [D_\xi]
\end{align*}
which is constant (and continuous) since the parallel transport between any  $\xi, \xi' \in T$ gives a relative $3$--chain bounding $D_\xi - D_{\xi'}$ in $S$.  Observe that the three edges of the rectangular domain of $D_\xi$ coming from $t=1$ and $s=\pm 1$, map to $p_0$, and the fourth edge, coming from $t=0$, winds around the equator $S^1$ via $\tau_S^2$.  So $\phi(\xi) \in H_2(S,S^1)$ should be non-zero because its image under the boundary map of a long exact sequence of the pair $H_2(S,S^1)\to H_1(S^1)$ is twice a generator. 

However,  we also have
\begin{align*}
D_{\iota(\xi)}&= \{\pr( \widetilde{h}_t(s\iota (\xi))) \; | \;  t\in [0,1] \text{ and } s\in [-1,1]  \}\\
			&= \{(\iota \circ \pr)( \widetilde{h}_t(s (\xi))) \; | \;  t\in [0,1] \text{ and } s\in [-1,1]  \}\\
			&= \iota( \{\pr( \widetilde{h}_t(s(\xi))) \; | \;  t\in [0,1] \text{ and } s\in [-1,1]  \})=\iota(D_\xi),
		\end{align*}
so we get $\phi(\iota(\xi))=-\phi(\xi)$ for any $\xi \in T$, but this is  only possible if $\phi(\xi)=0$ in $H_2(S, S^1)$ contradicting to the above.

By changing $\tau_R^2$ to $\tau_R^{2n}$ we can repeat arguments verbatim to show that $T_R^{2n} \neq 1$ in $\M(DTR)$, except now $\phi(\xi)$ is mapped to $2n$ times a generator of $H_2(S,  S^1)$.  If $T_R^n=1$ for some $n \in \Z^+$, then $T_R^{2n}=1$.  This proves that $T_R$ has infinite order in $\M(DTR)$.\footnote{Note that our arguments here in fact show that no non-trivial power of $\tau_R$ is homotopic to identity in its tubular neighborhood.}	

\noindent (c) We can argue as  in \cite{Konno} for the case of a Dehn twist.  The projective twist is supported in the neighborhood of $R$. We claim that the disk normal bundle of $R$ has unique spin structure relative to its unit circle bundle. Since the spin structures are in bijection with the first cohomology with $\Z_2$--coefficients, our claim follows from  the  calculation
$H^1(\nu R, \partial \nu R ;\Z_2)\cong H_3(\nu R;\Z_2) \cong H_3(R;\Z_2)=0$.
\end{proof}

We note that Proposition~\ref{prop:PT}(b)  does not imply anything on the order of a projective twist $T_R$ in $\M(X)$. In fact, it may even be a trivial mapping class; see e.g. Remark~\ref{rk:trivialPT} below. To ensure that we get a non-trivial mapping class, we are going to impose an algebraic condition on how $R$ sits in $X$:

\begin{defn}
Let $\RP\xhookrightarrow{i}  X$ be an embedding, where $X$ is an oriented $4$--manifold. We say the real projective plane $R:=i(\RP)$ is \emph{essential} in $X$ if  $i_* \colon \pi_1(\RP) \to \pi_1(X)$ has a retract.
\end{defn}

Suppose there is a retraction $r$, so the composition $\pi_1(\RP)\stackrel{i_*}{\to} \pi_1(X) \stackrel{r}{\to} \pi_1(\RP)$ is identity. There is a double cover $p \colon \widetilde{X} \to X$ corresponding to the kernel of $r$. Since $i_*$ is injective, the generator of $\pi_1(\RP)$ doesn't map into this kernel, so the generator of $\pi_1(R)$ is not in the image of $p_*(\pi_1(\widetilde{X}))$. It follows that $p^{-1}(\RP) \cong S^2$ in $\widetilde{X}$ and not $\RP \sqcup \RP$. Conversely, if there is a double cover $p \colon \widetilde{X} \to X$ such that $p^{-1}(R) \cong S^2$, by taking $r$ to be the composition $ \pi_1(X) \to \pi_1(X)/p_*(\pi_1(\widetilde{X})) \cong i_*(\RP)$ we get a retract. In short, the algebraic condition describing  the essentiality of $R$ in $X$ is equivalent to $X$ admitting a double-cover to which $R$ lifts as a sphere. 

In this case, we get a non-trivial mapping class:

\begin{proposition} \label{prop:essential}
Let $X$ be a closed oriented $4$--manifold with either $b^+(X) \neq b_1(X)-1$ or $
b^-(X) \neq b_1(X)$. Let $R$ be an essential real projective plane in $X$ with normal Euler number $\pm 1$. Then $T_R^m \neq 1$ in $\M(X)$,  for any odd integer $m$.  In particular, when $X$ has finite fundamental group, $T_R$ is non-trivial. 
\end{proposition}
 
\begin{proof}
Since $R$ is essential, there is a double cover $p\colon \widetilde{X} \to X$ such that $p^{-1}(R)=S$. Since the normal Euler number of $R$ is $\pm 1$, that of $S$ is $\pm 2$ in $\widetilde{X}$. Let $\tau' \in \D(X)$,  supported in a tubular neighborhood of $R$,  represent $T_R$,  then it is covered by some $\tau \in \D(\widetilde{X})$,  supported in a tubular neighborhood of $S$,  representing $T_S$; that is $\tau' \circ p = p \circ \tau$. 
We have the following commuting diagram:
	$$\begin{tikzcd}
		\widetilde{X} \arrow{r}{\tau}  \arrow{d}[left]{p} \arrow{dr}[below = 7, left =-5]{\tau' \circ p}& \widetilde{X}\arrow{d}{p}\\
		X \arrow{r}[below]{\tau'} & X. 
	\end{tikzcd}
	$$
 
Now assume for contradiction that $T_R=1$, so $\tau' \sim id_X$ in $\D(X)$. Then $\tau' \circ p \sim p$. Looking at the top-right triangle in the diagram, we can use homotopy lifting property of coverings to conclude  that $\tau \sim f$ for some $f \in \D(\widetilde{X})$ covering $id_X$. So $f$ is a deck transformation of the double cover $p$ and necessarily $f^2=id_{\widetilde{X}}$. And $f\neq id_{\widetilde{X}}$ because $f_*=\tau_*=(T_S)_*$ acts non-trivially on $[S] \neq 0$ in $H_2(\widetilde{X})$. So $f$ must be the non-trivial deck transformation.  Running the same arguments for $(\tau')^m$ with odd $m$ instead of $\tau'$,  we also get a lift $f \sim \tau^m$  which is the non-trivial deck transformation since $f_*=\tau_*=(T_S)^m_*=(T_S)_*$. 

Since $f$ is the non-trivial deck transformation of the covering $p \colon \widetilde{X} \to X$, we have $H^i(X)\cong H^i(\widetilde{X})^f$, the fixed subspace of $H^i(X)$ under the action of $f$. Since $f_*=(T_S)_*$, $f$ acts trivially on $H^1(\widetilde{X})$, so we get
$$b_1(X)=b_1(\widetilde{X}) ,$$ 
and by Proposition~\ref{prop:MT}(a), $f$ acts on $H^2(\widetilde{X})$ such that  we have $b^+_f(\widetilde{X})=b^+(\widetilde{X})$ and $b^-_f(\widetilde{X})=b^-(\widetilde{X})-1$, which implies
\begin{align*}
		b^+(X)&=b^+(\widetilde{X})\\
		b^-(X)&=b^-(\widetilde{X})-1 .
\end{align*}
Here $b_f^\pm$ denote the maximal dimensions of the positive-definite and negative-definite $f$-invariant subspaces of the free part of $H_2(X)$. Lastly, since $ \widetilde{X}$ is a double cover of $X$, we also have
\begin{align*}
\chi(\widetilde{X})&=2\chi(X)\\
\sigma(\widetilde{X})&=2\sigma(X)
\end{align*}
where $\chi$ denotes the Euler characteristic. 

Combining these equalities, we see that the topology of $X$ is quite restricted:
\begin{align*}
b^+(X)&=b_1(X)-1 \\
b^-(X)&=b_1(X). 
\end{align*}
By the hypotheses of the proposition, we conclude that $T_R^m \neq 1$ in $\M(X)$ for any odd integer $m$.
\end{proof}

\smallskip

\begin{remark} \label{rk:trivialPT}
The proof of the propositionshows that for an essential $R$,  we actually have $T_R \neq 1$ in the \emph{topological} mapping class group $\pi_0(\textrm{Homeo}^+(X))$.  However, for a non-essential $R$ this does not need  to be the case: Take for instance the standard embedding of $\RP$ in $ \CP$ as the fixed point set of the complex conjugation; it is Lagrangian with respect to the standard Fubini-Study K\"{a}hler form $\omega_{FS}$ on $ \CP$.  Since $T_R$  acts trivially on the homology by Proposition~\ref{prop:PT}(a),  it follows from \cite{Quinn} that $T_R=1$ in $\pi_0(\textrm{Homeo}^+(\CP))$.  And in fact,  by Gromov's work in  \cite{Gromov},  $T_R=1$ even in $\M(\CP, \omega_{FS})$, and thus in the smooth mapping class group $\M(\CP)$. 
\end{remark}

There are essential projective planes with normal Euler numbers $\pm 1$ in closed \mbox{$4$--manifolds}; we present a particularly interesting example in Lemma~\ref{lem:Hitchin}(b) in Section~\ref{sec:proofthmb} and further examples in the proof of Theorem~C in Section~\ref{sec:proofthmc}.

\smallskip
\subsection{Multi-reflections.}
Let $c$ be the complex conjugation $[z_0:z_1:z_2] \mapsto [\bar{z}_0:\bar{z}_1:\bar{z}_2]$ on the complex projective plane $\CP$, followed by an isotopy in the complement of  
$\CPo:=\{z_2=0\} \subset \CP$ taking a $4$--ball $D$ identically back to itself. Reversing the orientation of the $4$--manifold, this prescribes a mapping class in $\M(\CPb \setminus \nu D)$. 

If $S$ is an \textit{exceptional sphere} in $X$, namely, an embedded sphere with normal Euler number $-1$, then there is a tubular neighborhood $N$ of $S$ and an orientation-preserving diffeomorphism $\phi'' \colon N \to \CPb \setminus \nu D$.
So, extending $\phi''^{-1}\circ c \circ\phi''$ by identity on $X \setminus N$, we get a mapping class $R_{S} \in \M(X)$, studied under the name a \emph{reflection} along $S$ by Konno, Miyazawa and Taniguchi in \cite{KonnoEtal}.

\begin{defn} \label{def:MR}
Let $S=\sqcup_{i=1}^k S_i$ be a disjoint union of embedded spheres with normal Euler numbers $-1$ in an oriented $4$--manifold $X$. The \textit{multi-reflection} along $S$ is the  mapping class $R_{S}:= R_{S_1}\cdot \ldots \cdot R_{S_k}$ in $\M(X)$.
\end{defn}

As in the case of multi-twists, because $R_{S_i}$ and $R_{S_j}$ for each $i \neq j$ can be supported away from each other, the order of the factors in the definition of $R_S$ does not matter. 

\begin{remark}
By switching the orientations we can also define a multi-reflection $T_S$ along a disjoint union of embedded spheres $S_i$ with normal Euler numbers $+1$. While one can more generally reflect along a mix of spheres with normal Euler numbers $\pm 1$, we will not really utilize those mapping classes.
\end{remark}

We record a couple of basic properties of multi-reflections:

\begin{prop} \label{prop:MR}
Let $R_S \in \M(X)$ be the multi-reflection along $S=\sqcup_{i=1}^k S_i$ and $N$ a compact tubular neighborhood of $S$.  Set $e_i:=[S_i] \in H_2(X)$.
	\begin{enumerate}[\rm{(}a\rm{)}]
		\item $ (R_S)_*(a)=a+ 2\sum_{i=1}^k (a \cdot e_i) \,e_i $ for any $a \in H_2(X)$.
		\item $R_S$ has order $2$ in $\M(N)$, and therefore also in $\M(X)$.
	\end{enumerate}
\end{prop}

\begin{proof}
(a) We can identify $X\cong X' \#k \CPb$ and each $S_i$ with a  $\CPo$ in a distinct $\CPb$ summand. For $a=a'+\sum_{i=1}^k u_i e_i$, with $a' \in H_2(X')$ and $u_i\in \Z$, we have 
	\begin{align*}
		(R_S)_*(a)
		=a'-\sum_{i=1}^k u_i e_i= a-2\sum_{i=1}^k u_i e_i
		 =  a+2\sum_{i=1}^k (a\cdot e_i) e_i.
	\end{align*}
This is because $R_S$ acts trivially away from $S$ and the local model of the reflection $c$ conjugates $\CPo$, so $R_S$ maps each hyperplane class $e_i$ to $-e_i$.

\noindent (b) It is observed by Konno, Miyazawa and Taniguchi in \cite{KonnoEtal} that by a result of Giansiracusa, the square of a reflection is equal to identity in $\M(\CPb \setminus D) \cong \M(\CPb)$ and thus in any $\M(X)$. The same also holds for the multi-reflection, since $R_S^2=R_{S_1}^2 \cdot \ldots \cdot R_{S_k}^2=1$. On the other hand, $R_S$ acts non-trivially on the homology.  So it should have order $2$ in $\M(N)$.
\end{proof}

\smallskip
\begin{remark}
Unlike Dehn twists and projective twists, multi-reflections cannot be described as symplectic mapping classes. As for a general feel, note that a symplectic $4$--manifold $(X, \omega)$ that is not rational or ruled, any embedded sphere of normal Euler number $-1$ is homologous (up to sign) to a symplectic rational $(-1)$--curve \cite{Li}. The action of the reflection on this integral homology class is $e \mapsto -e$, but the canonical class, which should be preserved under a symplectomorphism,  pairs with $e$ non-trivially.
\end{remark}

\medskip
\section{Obstructions from Seiberg-Witten theory}

In Theorems~\ref{thm:obstruction1} and ~\ref{thm:obstruction2} below we are going to list some constraints on lifting subgroups generated by multi-twists and multi-reflections, respectively. These results and their proofs are direct generalizations of the ones given by Konno in \cite{Konno} and Konno, Miyazawa and Taniguchi in \cite{KonnoEtal} using different flavors of involutive Seiberg-Witten theory.

\begin{lemma} 
Let $X$ be a closed oriented spin $4$--manifold with $\sigma(X) <0$ and let $S=\sqcup_{i=1}^k S_k \subset X$ be a disjoint union of embedded spheres $S_i$,  where each one of the first $k_+$ of \,$\{S_1 , \ldots, S_k\}$ has normal Euler number $+2$ and of the remaining \,$k_-=k-k_+$ has normal Euler number $-2$. Suppose that $\dfrac{\sigma}{2} \neq k_+ - k_-$. If $k_+=0$ or $-\dfrac{\sigma}{16} +1 >k_+>0$, then the multi-twist $T_S \in \M(X)$ is not represented by any finite order element in $\D(X)$.
\end{lemma}

\begin{proof}
Suppose on the contrary that $f \in \D(X)$ has a finite order and is isotopic to the multi-twist $T_{S}$. We will rely on Theorem 1.3 of Konno \cite{Konno} to get a contradiction. To apply it, we need to compute $b_f^+$ and $b_f^-$, which are the maximal dimensions of the positive-definite and negative-definite $f$-invariant subspaces of $H_2(X,\R)$. 
  
  It follows from Proposition~\ref{prop:MT}(a) that for any $a \in H_2(X; \R)$,  we have 
  $$(T_{S})_*(a) = a + (a\cdot c_1) c_1 + \cdots + (a\cdot c_k)c_k$$
where we now take $c_i:= [S_i] \in H_2(X;\R)$. So the fixed subspace of $H_2(X;\R)$ is $$U:= \{a \in H_2(X;\R)\;|\; a\cdot c_i = 0 \; \forall i \}$$ 
because $c_i$ are linearly independent. Let $U_i:= \{a \in H_2(X;\R)\;|\; a\cdot c_i = 0  \}$, so we have $U=\cap_{i=1}^k U_i$. We claim that $U$ has dimension $b_2-k$, which we can prove by induction. The base step is immediate because $c_1^2\neq 0$, so it is essential, and the intersection form is non-degenerate. Suppose that $\cap_{i=1}^{k-1} U_i$ has dimension $b_2-(k-1)$. To prove the claim for $k$ it suffices to show that $\cap_{i=1}^{k} U_i \subset \cap_{i=1}^{k-1} U_i$ is a proper inclusion, i.e. $\cap_{i=1}^{k-1} U_i \not\subset U_{k}$. The latter holds, since $c_{k}\cdot c_i=0$ for all $i\leq k-1$, because the spheres are disjoint, and $c_{k}^2 \neq 0$.
		
Let $W:=\langle c_1, \dots, c_k \rangle$. Then $\dim W =k$, because if $x_1c_1+\cdots +x_kc_k=0$ for some $x_i \in \R$, then for each $i$ we have $c_i \cdot (x_1c_1+\cdots + x_kc_k)=0$, implying that $\pm 2 x_i = 0$,  so $x_i=0$ for all $i$.  This shows, by dimension reasons, that the subspaces $U$ and $W$ are complementary in $H_2(X;\R)$ and $U\cap W = 0$. By Sylvester's law of inertia, $b_f^+=b^+-k_+$ and $b_f^-=b^- -k_-$ and $\sigma_f=\sigma +k_--k_+$. 

Theorem~1.3 of \cite{Konno} then implies that we could not have $k_+=0$ or $-\dfrac{\sigma}{16} +1 >k_+>0$. Hence, we conclude that there is no finite order $f \in \D(X)$ realizing $T_S$.
\end{proof}

\smallskip
\begin{remark} \label{rk:positivesign}
When $\sigma(X)>0$, the statement of the lemma holds verbatim once we switch the roles of $k_+$ and $k_-$ and the sign of $\sigma(X)$.
\end{remark}

\begin{theorem} \label{thm:obstruction1}
Let $X$ be a closed connected $4$--manifold with finite fundamental group, signature $\sigma(X) <0$ (resp. $\sigma(X) >0$) and a spin universal cover. Let $S=\sqcup_{i=1}^k S_i$ be a disjoint union of embedded spheres, where each $S_i$ has normal Euler number $-2$  (resp.\,$+2$) and $k\neq -\sigma/2$ (resp. $k \neq \sigma/2)$. 
The multi-twist $T_S \in \M(X)$ is not represented by any finite order element in $\D(X)$. 
\end{theorem}

\begin{proof}
If $X$ is simply-connected (and thus its own universal cover) the result is immediate from the previous lemma. 
Let $p \colon \widetilde{X} \to X$ be the spin universal cover of degree $m= | \pi_1(X)|>1$. Each component $S_i$ of $S$ is a sphere, so $p^{-1}(S_i)$ is necessarily a disjoint union of $m$ embedded spheres, and one can take a tubular neighborhood of $S_i$ evenly covered by tubular neighborhoods of these spheres, all with the same normal Euler numbers as $S_i$. It follows that   $\widetilde{S}:=p^{-1}(S)$ is a disjoint union of $mk$ embedded spheres, each with normal Euler number $-2$, and there are $\widetilde{\tau} \in \D(\widetilde{X})$ and $\tau \in \D(X)$ 
representing the multi-twists $T_{\widetilde{S}} \in \M(\widetilde{X})$ and $T_S \in \M(X)$ such that $\tau \circ p = p \circ \tilde{\tau}$.

Now suppose $\tau$ is isotopic to $f \in \D(X)$ with order $n \in \Z^+$. The map $f$ is covered by some diffeomorphism $\tilde{f} \in \D(\widetilde{X})$, that is $f \circ p = p \circ \tilde{f}$. So $\tilde{f}^n$ is a deck transformation of the universal cover $p \colon \widetilde{X} \to X$. Since this desk transformation group is isomorphic to $\pi_1(X)$, we have $\tilde{f}^{nm}=1$. On the other hand, by the homotopy lifting property, $\widetilde{\tau}$ representing the multi-twist $T_{\widetilde{S}} \in \M(\widetilde{X})$ is homotopic to $\tilde{f}$, which contradicts the lemma {(for $k=k_-$)}.

For the parenthetical case ($\sigma(X)>0$ and all $S_i^2=+2$), the proof goes mutatis mutandis, where we invoke the observation in Remark~\ref{rk:positivesign}.
\end{proof}

We derive the following obstruction for projective twists:

\begin{corollary} \label{cor:obstruction}
Let $X$ be a closed connected $4$--manifold with finite fundamental group, signature $\sigma(X) < -1$ and (resp. $\sigma(X) > 1$) and a spin universal cover. Let $R$ be an essential projective plane in $X$ with normal Euler number $-1$ (resp. $+1)$. The projective twist $T_R \in \M(X)$ is not represented by any finite order element in $\D(X)$. 
\end{corollary}

\begin{proof}
Since $R$ is essential, there is a double cover $q \colon X' \to X$ such that $q^{-1}(R)$ is an embedded sphere $S$. Since $R$ has normal Euler number $-1$, that of $S$ should be $-2$.  

Now if $\tau' \in \D(X)$ supported near $R$ represents $T_R$, there is $\tau \in \D(X')$ supported near $S$ that covers it, and for $[\tau']=T_R$ we get $[\tau]=T_S$. By similar arguments as above, we see that if $\tau'$ were to be homotopic to a finite order element in $\D(X)$, then so would $\tau$ in $\D(X')$.  {Note that since $\sigma(X)<-1$,  we have $\sigma(X') < -2$,  so we get a contradiction to Theorem~\ref{thm:obstruction1} (with $k=1$).}

Switching the orientation of $X$, we get the mirror (parenthetical) claim. 
\end{proof}

\begin{remark} \label{rk:multiPT}
The corollary easily generalizes to give us an obstruction to lifting the projective multi-twist $T_R \in \M(X)$ for a disjoint union of essential real projective planes $R= \sqcup_{i-1}^k R_i$ all with the same normal Euler number $\pm 1$,  {provided $k \neq -\sigma$}.
\end{remark}

\smallskip
For the next obstruction, we need the following lemma: 

\begin{lemma}
Let $R_S$ be a multi-reflection along $S=\sqcup_{i=1}^k S_i$ in a closed oriented  $4$--manifold $X$ with $b_1(X)=0$ and $\sigma(X)\neq -2k$. Let $s \in \textrm{Spin}^c(X)$ such that \linebreak $c_1^2(s)-\sigma(X) > 0$ and $R_S^*(s) \cong \overline{s}$. Then $R_S$ is not homotopic to any smooth involution. 
\end{lemma}

\begin{proof}
Suppose that the multi-twist $R_S$ is homotopic to a smooth involution $f$. Then $f^*(s) \cong \overline{s}$. By Proposition~\ref{prop:MR}(a), $f$ acts on any $a \in H_2(X; \R)$ as
$$(f)_*(a)=a+ 2\sum_{i=1}^k (a \cdot e_i) \,e_i $$
for $e_i:=[S_i] \in H_2(X; \R)$.
Then
$$b^+_f (X)=b^+(X)$$ 
$$b^-_f (X) = b^-(X)-k$$
$$\sigma_f(X)=\sigma(X)+k.$$
Here $\sigma_f (X) \neq \sigma(X) / 2 \iff \sigma(X)+k \neq \sigma(X) / 2 \iff \sigma(X) \neq -2k$, which is what we assumed. Now, proceeding as in \cite{KonnoEtal}, since $f$ is odd type by \cite{KonnoEtal}[Lemma~2.10] and $b^+_f (X)=b^+(X)$,  we get a contradiction by Theorem 1.1 of \cite{KonnoEtal}. We conclude that $R_S$ cannot be homotopic to such $f$.
\end{proof}

\begin{theorem} \label{thm:obstruction2}
Let $X'$ be a closed oriented spin $4$--manifold with $b_1(X')=0$ and  $\sigma(X')<0$. Let $X= X' \# k \, \CPb$ and  $S \subset X$ be a disjoint union of {$k \neq -\sigma(X) / 2$} exceptional spheres. Then the  multi-reflection $R_S \in \M(X)$ is not represented by an order $2$ element in $\D(X)$.
\end{theorem}

\begin{proof}
Let $c$ be the Poincar\'{e} dual to $\sum_{i=1}^k [S_i]$, where $S_i$ are the components of $S$.
From 
\[
H^2(X) \cong H^2(X') \oplus \, H^2(\CPb) \oplus \cdots \oplus H^2(\CPb)
\]
and $X'$ having an even intersection form, we see that $c$ is characteristic and $c=c_1(s)$ for some $s \in \textrm{Spin}^c(X)$. From the action of $R_S$ on $H_2(X)$ one can see that $R_S^*(s) \cong \overline{s}$. 

On the other hand, we have
$$ c_1^2(s)-\sigma(X) = -k - \sigma(X')+k =-\sigma(X')> 0.$$
As we also have $b_1(X)=0$, using the previous lemma we conclude that $R_S$ is not homotopic to an involution in $\D(X)$.
\end{proof}

\medskip
\section{Proof of Theorem~A} \label{sec:proofthma}

In this section, we present the non-spin examples with indefinite intersection forms promised in Theorem~A.  All the mapping classes we are going to involve are multi-twists.

\begin{proof}[Proof of Theorem~A] We deal with the odd and even intersection forms separately. 

\noindent \underline{\textit{Odd forms}}:
Our building block is a closed oriented smooth $4$--manifold $T$ with finite fundamental group, signature $\sigma(T)=-1$ and spin universal (finite) cover $\widetilde{T}$; the existence of such a $4$--manifold is proved by Teichner in \cite{Teichner}.
We are going to thus refer to $T$ as \emph{Teichner's manifold}.\footnote{The existence of $T$ follows from a spectral sequence calculation in the spin bordism groups in \cite{Teichner}; in particular, one does not have further information on $b_2(T)$. We can of course take our $T$ to be (one of the smooth $4$--manifolds) with the smallest $b_2$ among possible candidates.} The fundamental group of $T$ is $\pi_1(T)= \Z_{16} \rtimes \Z_8$ with action of $\Z_8$ on $\Z_{16}$ given by $t \to t^5$; in particular $\pi_1(T)$ admits a presentation with two generators. Let $W$ be a wedge of two circles in $T$ which represent these two generators. 

We can now take a connected sum of two copies of $T$ along $W$, which is a particular case of a \emph{connected sum along $1$--skeletons} \cite{Kreck}. We get a new closed oriented smooth $4$--manifold $X_2:= (T \setminus \nu W) \cup (T \setminus \nu W)$, by gluing the boundaries with an orientation-reversing diffeomorphism $\phi$ of $\partial (\nu W) \cong S^1 \x S^2 \, \# \, S^1 \x S^2$, say, by taking $\rm{id}_{S^1} \x r$ on each $S^1 \x S^2$ summand, where r is a reflection on $S^2$. 

This $4$--manifold $X_2$ enjoys a few properties that are of importance to us. First, by Seifert-Van Kampen, $\pi_1(X_2)= \pi_1(T)$, generated by a push-off of $W$ into either copy of $T \setminus \nu W$. Second, we can see that $X_2$ also has a (finite) spin cover, as follows: Given the spin universal cover $p \colon \widetilde{T} \to T$, let $U:= p^{-1}(T \setminus \nu W)$. As $W$ contains the $\pi_1(T)$ generators, $\partial U = p^{-1}(\partial \nu W)$ is connected, and $\pi_1(U) \cong \pi_1(\widetilde{T}) \cong 1$. Therefore we get a spin universal cover $\widetilde{X}_2:=U \cup_{\widetilde{\phi}} U$ of $X_2$, as we glue the boundaries by an orientation-reversing diffeomorphism $\widetilde{\phi}$ covering $\phi$ (the existence of which is easy to see due to our choice of $\phi$ above).  Third, because $\nu W$ has the homotopy type of a $1$--skeleton, the intersection form $Q_{X_2}$ is also odd and we have $\sigma(X_2)=2 \sigma(T)=-2$ by Novikov additivity.

With the above in mind, we can now inductively define a closed oriented smooth $4$--manifold $X_s$ by taking a connected sum of $|s|$ copies of $T$ along $W$, so that $\pi_1(X_s)=\pi_1(T)$, $X_s$ has a spin universal (finite) cover $\widetilde{X}_s$, and signature $\sigma(X_s)=s$. A straightforward calculation shows that $b_2(X_s)= |s| (b_2(T)+4)-4$. For $s>0$, we run the same construction to build $X_s$, this time using the building block $\overline{T}$, Teichner's manifold with the reversed orientation. The algebraic topology of $X_s$ for $s>0$ reads the same as above. 

Now, for any integer $s \neq 0$ and each $n \in \Z^+$, we take $X_{s,n}:=X_s \ \#_n \, (S^2\times S^2)$, the connected sum of $X_s$ with $n$ copies of $S^2 \x S^2$. The closed oriented smooth $4$--manifold $X_{s,n}$ has $\pi_1(X_{s,n})=\pi_1(T)$, a spin universal (finite) cover $\widetilde{X}_{s,n} \cong \widetilde{X}_s \#_{128 n} (S^2\times S^2)$, $\sigma(X_{s,n})=s$, and an odd, indefinite intersection form of rank $|s| (b_2(T)+4)-4+2n$. 

Hence, for any given $s$, the family of $4$--manifolds $\{ X_{s,n} \, | \, n \in \Z^+\}$ realizes all but finitely many odd indefinite unimodular integral quadratic forms. (As the signature is fixed, the rank of the intersection form increases by $2$ each time, realized by an additional $S^2 \x S^2$ summand.) The additional $S^2 \x S^2$ summands contain many disjoint emedded spheres with normal Euler numbers $\pm 2$, e.g. the diagonals and anti-diagonals of each $S^2 \times S^2$ summand. Taking any $S$ that is a disjoint union of these spheres (with normal Euler numbers $-2$ if $s<0$ and $+2$ if $s>0$, and $\pi_0(S) \neq |\sigma|/2$), we get an order two subgroup $\langle T_S \rangle$ of $\M(X_{s,n})$ (by Proposition~\ref{prop:MT}(b)) that does not lift to $\D(X_{s,n})$ by Theorem~\ref{thm:obstruction1}.

\smallskip
\noindent \underline{\textit{Even forms:}} 
Our construction will be similar to the one we gave in the odd case, but this time our building block is the \textit{Enriques surface} $E$, which has $\pi_1(E)=\Z_2$, spin universal cover the $K3$ surface, and even intersection form $Q_E \cong E_8 \oplus H$. Here $E_8$ denotes the standard even negative-definite unimodular integral quadratic form of rank eight and $H$ denotes the even indefinite hyperbolic form of rank two, i.e. $H \cong Q_{S^2 \x S^2}$. Let $V$ denote a loop that generates the $\pi_1(E)$.

Recall that an even indefinite unimodular integral quadratic form $Q$ with non-zero signature is isomorphic to $\pm r\, E_8 \oplus q \, H$ for some $r \in \Z^+$, $q \in \Z$. First assume $s=8r <0$. Similar to before, we first build the closed oriented smooth $4$--manifold $Y_r$, which is a connected sum of $|r|$ copies of $E$ along $V$, so that $\pi_1(Y_r)=\pi_1(E)$, $Y_r$ has a spin universal (finite) cover, and signature $\sigma(Y_r)=s$. If $s>0$, we run the same construction to build $Y_r$, this time using the building block $\overline{E}$. Now consider $Y_{r,n}:= Y_r \ \#_n S^2 \x S^2$, for any $n \in \N$, which also have $\pi_1(Y_r)=\Z_2$, a spin universal cover, $\sigma(Y_{r,n})=-8r$, and an even intersection form of rank $12|r|+2(n-1)$. 

For any given $s=8r \neq 0$, the family of $4$--manifolds $\{ Y_{r,n} \, | \, n \in \N\}$ realizes all but finitely many even unimodular integral quadratic forms. In this case, not only the additional $S^2 \x S^2$ summands contain many disjoint embedded spheres with normal Euler numbers $\pm 2$, but also in the $Y_r$ summand, each $E \setminus \nu V$ (resp. $\overline{E \setminus \nu V}$) contains spheres of normal Euler number $-2$ (resp. $+2$) as well. (For instance, there are non-generic complex Enriques surfaces, such as nodal Enriques surfaces, containing up to eight disjoint holomorphic rational $(-2)$--curves. Since $V$ can be taken to be disjoint from them, they contribute many more spheres with normal Euler numbers $-2$ in each $E \setminus \nu V$ summand. Another rational $(-2)$--curve is described in Lemma~\ref{lem:Hitchin}(b).) 
Once again, taking any $S$ that is a disjoint union of these spheres (with signs of their normal Euler numbers and the number of components satisfying the same conditions as before), we once again get an order two subgroup $\langle T_S \rangle$ of $\M(Y_{r,n})$ that does not lift to $\D(Y_{r,n})$ by Theorem~\ref{thm:obstruction1}.

\smallskip
\noindent \underline{\textit{Irreducible $4$--manifolds}}: Among the $4$--manifolds $X_{s,n}$ and $Y_{r,n}$ we built above, there is  one pair of irreducible $4$--manifolds, $Y_{\pm 1,0}$, the Enriques surface $E$ and its mirror. (They are all reducible for $n>0$, and for any other $Y_{r, 0}$ with $r \neq \pm 1$, we don't have access to any geometric tools to examine irreducibility ---however, one can observe that their universal covers have trivial Seiberg-Witten invariants.)

Here is a large family of irreducible examples we can add to our list: Let $X$ be any complex elliptic surface with $t>0$ many $I_2$ fibers, with finite fundamental group of order $m >1$, and holomorphic Euler characteristic $\chi_h=n >0$ such that either $m$ or $n$ is even. Each $I_2$ fiber contains a rational $(-2)$-curve and we can take $S \subset X$ to be the disjoint union of these $t$ embeddes spheres with normal Euler numbers $-2$. We claim that the order two subgroups $\langle T_S \rangle$ of $\M(X)$ does not lift to $\D(X)$. This is because, the universal cover $\widetilde{X}$ is a minimal complex surface with  $$c_1^2(\widetilde{X})=m c_1^2(X)=m\, 0 =0 \ \ \text{ and }$$ $$\chi_h(\widetilde{X})=m \chi_h(X)=mn \, .$$
Therefore, $\widetilde{X}$ is also a complex elliptic surface, and since it is simply-connected and $mn$ is even, it has the same homotopy type as the spin elliptic surface $E(mn)$. So we can once again invoke Theorem~\ref{thm:obstruction1} to obstruct the liftability. Furthermore, we can get examples with both odd and even intersection forms among these. Say $X$ is obtained from a complex $E(n)$ with $t>0$ $I_2$ fibers by a pair of logarithmic $p >1$ transforms. Then, it is known that $\pi_1(X)= \Z_p$ and the intersection form $Q_X$ is even if $n$ is even and $p$ is odd, and $Q_X$ is odd otherwise. Hence we get the promised examples, since all these elliptic surfaces are irreducible.
\end{proof}

\begin{remark}\label{rk:irreducible}
In the spin case,  an analog of Theorem~A holds by Konno's work in \cite{Konno}, and it is easy to see that there are vast families of irreducible spin $4$--manifolds for which Konno's results are applicable; e.g. one can take any spin K\"{a}hler surface with non-zero signature and containing rational $(-2)$--curves. Even more irreducible examples we can point to directly are available in the symplectic category: In \cite{ArabadjiBaykur}, for any prescribed finitely presented group $G$ (including non-K\"{a}hler groups) we have constructed a spin symplectic Lefschetz fibration $X_G$ with $\sigma (X_G) <0$. If needed, adding two more (untwisted) fiber sum summands, one can guarantee the existence of disjoint embedded spheres with normal Euler number $-2$, each given by a pair of matching vanishing cycles. By Taubes,  these spin (thus minimal) symplectic \mbox{$4$--manifolds} are all irreducible. 
\end{remark}

\medskip
\section{Proof of Theorem~B}\label{sec:proofthmb}

One of the main building block we are going to use here is the \textit{Hitchin manifold} $H$, which is a closed oriented Einstein manifold that arises as a quotient of a complex Enriques surface by a free anti-holomorphic involution \cite{Hitchin}.

\begin{lemma} \label{lem:Hitchin}
The Hitchin manifold $H$ satisfies the following:
\begin{enumerate}[\rm{(}a\rm{)}]
\item $H$ is irreducible and has negative-definite intersection form $Q_H \cong 4\langle -1 \rangle$.
\item  $H$ contains an essential projective plane $R$ with normal Euler number $-1$.   
\end{enumerate}
\end{lemma}

\begin{proof} (a) The algebraic topology of the closed oriented smooth $4$--manifold $H$ is easily inferred from that of an Enriques surface and is readily described by Hitchin in \cite{Hitchin}: $\pi_1(H)=\Z_2 \oplus \Z_2$, $b^+(H)=0$ and $\sigma(H)=-4$, so $b^-(H)=4$ and by Donaldson's diagonalization theorem, $Q_H \cong 4\langle -1 \rangle$. Since $H$ is finitely covered by an Enriques surface, which is irreducible, $H$ too is irreducible.

\noindent(b) For the existence of an essential projective plane $R$ we are going to turn to real algebraic geometry: A \textit{real Enriques surface} is a complex Enriques surface with an anti-holomorphic involution on it. The deformation classes of real Enriques surfaces are well-understood, and in the case that is of interest to is, namely when the involution is free, the moduli space is connected \cite{DegtyarevEtal}. This means in particular that the diffeomorphism type of $H$ is independent of the complex Enriques surface admitting an anti-holomorphic involution. 

Therefore, to find an essential projective plane $R$ in $H$ it suffices to find a \textit{particular} real Enriques surface $E$ with a free anti-holomorphic involution $c$ on it which fixes a holomorphic rational $(-2)$-curve $C$ in $E$. (In turn, this translates to finding a pair of rational $(-2)$--curves in a non-generic $K3$ surface $K$ covering this $E$.)
Since the normal Euler number of $C$ is $-2$, that of $R$ will be $-1$, and as we noted earlier, because $R$ lifts to a sphere in a double cover, it is essential. 

The following argument for the existence of the desired pair of $E$ and $C$ is kindly provided by Alex Degtyarev, and is written following the techniques, results and notation of \cite{DegtyarevEtal}.

\noindent \underline{\textit{Degtyarev's lattice theoretic construction}}:
Let $K \to E$ be the covering $K3$ surface and let $\tau$ be the deck transformation of this double cover. Let $c^{i} \colon K \to K$, $i=1,2$ denote the two lifts of the real structure $c$ on $E$, so that $\tau=c^1\circ c^2$. We have the lattice $L:=H_2(K)\cong 2 \mathbf{E}_8\oplus 3 \mathbf{U}$. (All root systems are negative-definite.)  A crucial observation is that the desired $E$ is of type $(I_0,I_0)$, i.e.
both halves of $E_\R$ (which are both empty) represent $0\in H_2(E;\Z_2)$.

Let $L^{\epsilon^1,\epsilon^2}\subset L$,
$\epsilon^{i}=\pm$, be the bi-eigenlattices of~$c^{i}$. Then by a computation in 
 \cite{DegtyarevEtal}[Section 21.3] of the Betti numbers and parity, we get
\[*
L^{++}\cong\mathbf{D}_4(2),\qquad L^{+-}\cong L^{-+}\cong L^{--}\cong\mathbf{D}(2)\oplus\mathbf{U}(2).
\]
According to \cite{DegtyarevEtal}[21.3], the lattices $L^{-+}$ and $L^{--}$ are glued together (to
the \linebreak $(-1)$-eigenlattice $L^{-1}$ of $c^1$, which is
$\mathbf{E}_8(2)\oplus\mathbf{U}(2)\oplus\mathbf{U}$) as follows: the minimal gluing (we can
say that $\mathbf{D}_4\oplus\mathbf{D}_4$ extends to $\mathbf{E}_8$) to
$\mathbf{E}_8(2)\oplus\mathbf{U}(2)\oplus\mathbf{U}(2)$ is followed by the index~$2$ extension
via\ $r^+:=\frac12(a^++a^-)$, where $a^\epsilon:=u_1^\epsilon-u_2^\epsilon$ and
$u_1^\epsilon,u_2^\epsilon$ is a standard basis for the $\mathbf{U}(2)$ summand
coming from $L^{-,\epsilon}$. Let $r^-:=\frac12(a^+-a^-)$. Then, both
$r^\pm\in L^{-1}$ and, since $(a^\pm)^2=-4$, we have
\[*
(r^\pm)^2=-2,\quad r^+\cdot r^-=0.
\]

Now, we have two orthogonal square $(-2)$ vectors $r^\pm$ such that
\[
c^1\:r^\pm\mapsto-r^\pm,\qquad \tau\:r^+\leftrightarrow r^-.
\label{eq.r}
\]
To choose the complex structure on~$K$ and thus~$E$, we choose the period (see
\cite{DegtyarevEtal}[Section~14] orthogonal to $a^\pm$ and generic otherwise (i.e. not orthogonal to
an integral $(-2)$-vector other than $\pm r^\pm$). The general theory of
$K3$-surfaces implies that $r^\pm$ (or $-r^\pm$)
are realized by a pair of disjoint $(-2)$--curves $C^\pm\subset X$;
by~\eqref{eq.r}, they are $c^1$-real and transposed by~$\tau$.
Moreover, since the
period is chosen generic, $K$ has no other $(-2)$--curves and, hence, $C^\pm$ are
both smooth rational. They project to a $(-2)$-curve $C$ in~$E$.
\end{proof}

We can now present our examples with definite intersection forms. 

\begin{proof}[Proof of Theorem~B] Consider  $Z_{n}:=  H \# \,n\, \CPb$, for any $n \in \Z^+$, where  $H$ is the Hitchin manifold. By Lemma~\ref{lem:Hitchin}(a), the closed oriented \mbox{$4$--manifold} $Z_{n}$ has a negative-definite intersection form $Q_{Z_{n}} \cong (4+n)\langle -1\rangle$. It moreover contains $n$ disjoint spheres $S_j$ with normal Euler numbers $-1$. 
Let $S$ be a disjoint union of {$k \neq (4+n)/2$} exceptional spheres among $\{S_j\}$, respectively. 

By Theorem~\ref{thm:obstruction2}, the multi-reflection $T_S$ does not lift, and $T_S$ is not isotopic to identity by Proposition~\ref{prop:MT}(b). Therefore the Nielsen realization fails for the  order two subgroup $\langle T_S \rangle$ of $\M(Z_{n})$. Switching the orientation, and invoking the mirror statement in Theorem~\ref{thm:obstruction2}, we get similar examples in $\M(\overline{Z}_{n})$, where now \linebreak $Q_{\overline{Z}_{n}} \cong (4+n)\langle 1 \rangle$.
\end{proof}

\smallskip
\begin{remark}\label{rk:rank4}
By Lemma~\ref{lem:Hitchin}(b), the Hitchin manifold $H$ contains an  essential projective plane $R$, so $\langle T_R \rangle$ is nontrivial in $\M(H)$. While we have not been able to determine whether $T_R$ is of finite order in $H$ or not, we note that, if it does, then by Corollary~\ref{cor:obstruction}, $H$ is an example of an \emph{irreducible} definite $4$--manifold with a finite subgroup in $\M(H)$ that does not lift to $\D(H)$.
\end{remark}

\begin{rmk}
One can leverage multi-twists and multi-reflections to generate many more examples of finite subgroups  of mapping class groups of $4$--manifolds that do not lift.  Assume that the multi-twist or multi-reflection along $S=\sqcup_{i=1}^k S_i$ in $X$ is not liftable.  For any $m \leq k$, taking $G=\langle T_{S_1}, \ldots, T_{S_m}\rangle$ or $G=\langle R_{S_1}, \ldots, R_{S_m}\rangle$, we get a non-liftable subgroup $G \cong \oplus \Z_2^{m}$ of $\M(X)$.  More generally,  assume that there is a finite subgroup $G'<\M(X)$ where each $g \in G'$ has a representative in $\D(X)$ permuting $\{S_i\}$.  Then, there is another finite subgroup $G'' \cong \Z_2^m \x G'$ of $\M(X)$ that does not lift to $\D(X)$.  
\end{rmk}

\medskip
\enlargethispage{0.25in}
\section{Proof of Theorem~C} \label{sec:proofthmc}

In this final section, we show that projective twists along essential Lagrangian projective planes give rise to examples of elements of the symplectic Torelli group that cannot be expressed as a product of squared Dehn twists along Lagrangian spheres.

\begin{proof}[Proof of Theorem~C] Start with a complex elliptic surface $Y=E(1)$ with two nodal fibers.  Symplectically blowing-up at one of the two nodes and then rationally blowing down the resulting $(-4)$-curve we get a new symplectic $4$--manifold $Y'$.  Repeating this  at the other node,  we derive another symplectic $4$--manifold $Y''$.  As observed by Fintushel and Stern \cite[Theorem~3.1]{FSrationalblowdown},  this special construction amounts to performing a logarithmic transform of order two on a regular fiber of $E(1)$ and thus $Y' \cong E(1)_2$ and $Y'' \cong E(1)_{2,2}$.  Note that $\pi_1(Y') \cong 1$ and $\pi_1(Y'') \cong \Z_2$.  

A symplectic rational blowdown along a \mbox{$(-4)$--curve} can be realized by symplectic summing with a quadric $Q$ in $\CP$ \cite{Gompf}.  Since the complement $\CP \setminus \nu Q$ contains the standard Lagrangian $\RP \subset \CP$,  in the  symplectic $4$--manifold obtained after the rational blowdown we have an embedded Lagrangian projective plane with normal Euler number $-1$. Let $R \subset Y''$ be such a Lagrangian projective plane  obtained after rationally blowing down the $(-4)$--curve $C$ in $Y'$.  

We claim that $R$ is essential in $Y''$.  To prove it,  let us assume the opposite.  By Seifert-Van Kampen, we have
	$$\begin{tikzcd}
		\pi_1(\partial (\nu R)) \arrow{r}  \arrow{d}{i_*} & \pi_1(\nu R) \arrow{d}[left]{0}\\ \pi_1(Y''\setminus \nu R)
		  \arrow{r} & \pi_1(Y'') 
	\end{tikzcd}
	$$
and so $ \coker{i_*} \cong \pi_1(Y'') \cong \Z_2$.  On the other hand,  we have
	$$\begin{tikzcd}
		\pi_1(\partial (\nu C)) \arrow{r}  \arrow{d}{j_*} & \pi_1(\nu C)\arrow{d}[left]{0} =0 \\
		\pi_1(Y'\setminus \nu C)
		\arrow{r} & \pi_1(Y')=0
	\end{tikzcd}
	$$
and so $ \coker{j_*}\cong 0$.   However,   $Y''\setminus \nu R=Y' \setminus \nu C$ and $\partial (\nu R) = \partial (\nu C)$,  thus  $i_*=j_*$ could not have different cokernels.  It follows that the inclusion induced homomorphism $\pi_1( \nu R) \to \pi_1(Y'') \cong \Z_2$ is a non-zero map,  and $R$ is essential in $Y''$. 

Therefore,  by Proposition~\ref{prop:essential},  $T_R^m \neq 1$ in $\M(Y'')$ for any odd integer $m$.
	

For more examples,  we can repeat the arguments above verbatim for the complex elliptic surface $Y=E(n)$,   so $Y''=E(n)_{2,2}$ and $\widetilde{Y''}=E(2n)$.  In  fact, $Y'' \cong E(1)_{2,2}$ is diffeomorphic to the Enrique surface $E$, which has symplectic Kodaira dimension zero, and for $n>1$, we get examples with symplectic Kodaira dimension one.  Taking symplectic fiber sum of $Y=E(n)$ along a disjoint copy of the elliptic fiber with a (simply connected) minimal symplectic $4$--manifold of general type,  we get examples with symplectic Kodaira dimension two.  (The latter building blocks are found in abundance; see e.g. \cite{BaykurEtal}.)

Since squared Dehn twists are trivial in the smooth mapping class group,  none of the projective twists above can be expressed as a product of them.
\end{proof}
    
\begin{remark}
In the above examples,  we can moreover observe that $T_R^2  \neq 1$ in $\M(X, \omega)$ as well.   For instance,  the  universal double  cover of $X=E(1)_{2,2}=E$, the  Enrique surface,   is diffeomorphic to the $\K$ surface.  Take the symplectic form $\widetilde{\omega}$· on $\K$ to be the one induced by the symplectic form $\omega$ on $E$ via the double cover.  Then $S:=q^{-1}(R)$ is a Lagrangian sphere in $\K$.  However,  as shown by Seidel in \cite[corollary 2.9]{Seidel2}, neither a Dehn twist nor its square in a closed minimal symplectic $4$--manifold $Z$ with $H^1(Z ; \R)=0$ and $b_2(Z) \geq 3$ can be symplectically isotopic to identity.  Therefore,  $T_R^2 \neq 1$ in $\M(E, \omega)$.  The same argument shows that the order of $T_R$ is not $2$ in the symplectic mapping class group of any other $(X, \omega)$ above.  
\end{remark}

\vspace{0.1in}
\noindent \textit{Acknowledgements.} This work was  supported by the  NSF grant  DMS-2005327. The second author would like to thank Harvard University and  Max Planck Institute for Mathematics in Bonn for their hospitality during the writing of this article. We are grateful to Alex Degtyarev for answering an inquiry of ours on the existence of a  specific real Enriques surface and contributing the argument on its existence,  which is  included in the proof of Lemma~\ref{lem:Hitchin}(b).

\end{document}